\documentclass[11pt]{amsart}
\usepackage{color,amssymb}
\usepackage[normalem]{ulem}

\usepackage[all]{xy}

\newtheorem{theorem}{Theorem}[section]
\newtheorem{claim}[theorem]{Claim}

\newtheorem{proposition}[theorem]{Proposition}
\newtheorem{corollary}[theorem]{Corollary}
\newtheorem{lemma}[theorem]{Lemma}

\theoremstyle{definition}
\newtheorem{remark}[theorem]{Remark}
\newtheorem{definition}[theorem]{Definition}
\newtheorem*{definition*}{Definition}

\newcommand{\IN}{\mathbb N}

\newcommand{\F}{\mathcal F}
\newcommand{\U}{\mathcal U}
\newcommand{\C}{\mathcal C}

\newcommand{\E}{\mathcal E}

\newcommand{\e}{\varepsilon}
\newcommand{\w}{\omega}

\newcommand{\Fil}{\mathrm{\varphi}}
\newcommand{\Ra}{\Rightarrow}
\newcommand{\Haus}{\mathsf{T_{\!2}S}}

\newcommand{\Zero}{\mathsf{T_{\!z}S}}

\newcommand{\Tau}{\mathcal T}

\newcommand{\korin}[2]{\!\sqrt[#1]{\!#2}}
\newcommand{\TS}{\mathcal T_{\mathcal D}}

\usepackage{relsize}
\newcommand*{\defeq}{\stackrel{\mathsmaller{\mathsf{def}}}{=}}

\newcommand{\Lambdae}{\Lambda^{\!e}}

\title{Injectively and absolutely $\mathsf{T_{\!1}S}$-closed semigroups}
\author{Taras Banakh}
\address{T.Banakh: Ivan Franko National University of Lviv (Ukraine) and Jan Kochanowski University in Kielce (Poland)}
\email{t.o.banakh@gmail.com}
\makeatletter
\@namedef{subjclassname@2020}{%
  \textup{2020} Mathematics Subject Classification}
\makeatother

\subjclass[2020]{22A15, 20M18, 54B30, 54D35, 54H11, 54H12}
\keywords{commutative semigroup, semilattice, group,  $\C$-closed semigroup, injectively $\C$-closed semigroup, absolutely $\C$-closed semigroup}

\begin{document}
\begin{abstract}  A semigroup $X$ is  {\em absolutely} (resp. {\em injectively}) $\mathsf{T_{\!1}S}$-{\em closed} if for any (injective) homomorphism $h:X\to Y$ to a $T_1$ topological semigroup $Y$, the image $h[X]$ is closed in $Y$. We prove that a commutative semigroup $X$ is injectively $\mathsf{T_{\!1}S}$-closed  if and only if $X$ is bounded, nonsingular and Clifford-finite. Using this characterization, we prove that (1) every injectively $\mathsf{T_{\!1}S}$-closed semigroup has injectively $\mathsf{T_{\!1}S}$-closed center, and (2) every absolutely $\mathsf{T_{\!1}S}$-closed semigroup has finite center. As a by-product of the proof we elaborate the technique of topologization of semigroups by remote bases.
\end{abstract}
\maketitle

\section{Introduction and Main Results}

In many cases,  completeness properties of various objects of General Topology or  Topological Algebra can be characterized externally as closedness in ambient objects. For example, a metric space $X$ is complete if and only if $X$ is closed in any metric space containing $X$ as a subspace. A uniform space $X$ is complete if and only if $X$ is closed in any uniform space containing $X$ as a uniform subspace. A topological group $G$ is Ra\u\i kov complete  if and only if it is closed in any topological group containing $G$ as a subgroup.

On the other hand, for topological semigroups there are no reasonable notions of (inner) completeness. Nonetheless we can define many completeness properties of semigroups via their closedness in ambient topological semigroups.

A {\em topological semigroup} is a topological space $X$ endowed
with a continuous associative binary operation $X\times X\to
X$, $(x,y)\mapsto xy$.

\begin{definition*} Let $\C$ be a class of topological semigroups.
A topological
semigroup $X$ is called
\begin{itemize}
\item {\em $\C$-closed} if for any isomorphic topological
embedding $h:X\to Y$ to a topological semigroup $Y\in\C$
the image $h[X]$ is closed in $Y$;
\item {\em injectively $\C$-closed} if for any injective continuous homomorphism $h:X\to Y$ to a topological semigroup $Y\in\C$ the image $h[X]$ is closed in $Y$;
\item {\em absolutely $\C$-closed} if for any continuous homomorphism $h:X\to Y$ to a topological semigroup $Y\in\C$ the image $h[X]$ is closed in $Y$.
\end{itemize}
\end{definition*}

For any topological semigroup we have the implications:
$$\mbox{absolutely $\C$-closed $\Ra$ injectively $\C$-closed $\Ra$ $\C$-closed}.$$

\begin{definition*} A semigroup $X$ is defined to be ({\em injectively, absolutely}) {\em $\C$-closed\/} if so is $X$ endowed with the discrete topology.
\end{definition*}

We will be interested in the (absolute, injective) $\C$-closedness for the classes:
\begin{itemize}
\item $\mathsf{T_{\!1}S}$ of topological semigroups satisfying the separation axiom $T_1$;
\item $\Haus$ of Hausdorff topological semigroups;
\item $\Zero$ of Tychonoff zero-dimensional topological
semigroups.
\end{itemize}
A topological space satisfies the separation axiom $T_1$ if all its finite subsets are closed.
A topological space is {\em zero-dimensional} if it has a base of
the topology consisting of {\em clopen} (=~closed-and-open) sets.

Since $\Zero\subseteq\Haus\subseteq\mathsf{T_{\!1}S}$, for every semigroup we have the implications:
$$
\xymatrix{
\mbox{absolutely $\mathsf{T_{\!1}S}$-closed}\ar@{=>}[r]\ar@{=>}[d]&\mbox{absolutely $\mathsf{T_{\!2}S}$-closed}\ar@{=>}[r]\ar@{=>}[d]&\mbox{absolutely $\mathsf{T_{\!z}S}$-closed}\ar@{=>}[d]\\
\mbox{injectively $\mathsf{T_{\!1}S}$-closed}\ar@{=>}[r]\ar@{=>}[d]&\mbox{injectively $\mathsf{T_{\!2}S}$-closed}\ar@{=>}[r]\ar@{=>}[d]&\mbox{injectively $\mathsf{T_{\!z}S}$-closed}\ar@{=>}[d]\\
\mbox{$\mathsf{T_{\!1}S}$-closed}\ar@{=>}[r]&\mbox{$\mathsf{T_{\!2}S}$-closed}\ar@{=>}[r]&\mbox{$\mathsf{T_{\!z}S}$-closed.}
}
$$

$\C$-Closed topological groups for various classes $\C$ were investigated by many authors
~\cite{AC,AC1,Ban,DU,G,L}. In particular, the closedness of commutative topological groups in the class of Hausdorff topological semigroups was investigated
in~\cite{Z1,Z2}; $\mathcal{C}$-closed topological semilattices were investigated
in~\cite{BBm, BBh, BBR, GutikPagonRepovs2010, GutikRepovs2008, Stepp69, Stepp75}.
For more information about complete topological semilattices and pospaces, see the  survey~\cite{BBc}.
This paper is a continuation of the papers \cite{CCUS}, \cite{BB}, \cite{BB2}, \cite{GCCS}, \cite{ACS},  providing inner characterizations of various closedness properties of (discrete topological) semigroups. In order to formulate such inner characterizations, let us recall some properties of semigroups.


A semigroup $X$ is called
\begin{itemize}
\item {\em unipotent} if $X$ has a unique idempotent;
\item {\em chain-finite} if any infinite set $I\subseteq X$ contains elements $x,y\in I$ such that $xy\notin\{x,y\}$;
\item {\em singular} if there exists an infinite set $A\subseteq X$ such that $AA$ is a singleton;
\item {\em periodic} if for every $x\in X$ there exists $n\in\IN$ such that $x^n$ is an idempotent;
\item {\em bounded} if there exists $n\in\IN$ such that for every $x\in X$ the $n$-th power $x^n$ is an idempotent;
\item {\em group-finite} if every subgroup of $X$ is finite;
\item {\em group-bounded} if every subgroup of $X$ is bounded.
\end{itemize}
We recall that an element $x$ of a semigroup is an {\em idempotent} if $xx=x$.

The following theorem (proved in \cite{BB}) characterizes $\C$-closed commutative semigroups.

\begin{theorem}[Banakh--Bardyla]\label{t:C-closed} Let $\C$ be a class of topological semigroups such that $\mathsf{T_{\!z}S}\subseteq\C\subseteq \mathsf{T_{\!1}S}$. A commutative semigroup $X$ is $\C$-closed if and only if $X$ is chain-finite, nonsingular,  periodic, and group-bounded.
\end{theorem}

For unipotent semigroups, Theorem~\ref{t:C-closed} was simplified in \cite{CCUS} as follows.

\begin{theorem}[Banakh--Vovk]\label{t:unipotent-C}  Let $\C$ be a class of topological semigroups such that $\mathsf{T_{\!z}S}\subseteq\C\subseteq \mathsf{T_{\!1}S}$. A unipotent commutative semigroup $X$ is $\C$-closed if and only if $X$ is bounded and nonsingular.
\end{theorem}

Theorem~\ref{t:C-closed} implies that each subsemigroup of a $\C$-closed commutative semigroup is $\C$-closed. On the other hand, quotient semigroups of $\C$-closed commutative semigroups are not necessarily $\C$-closed, see Example 1.8 in \cite{BB}. This motivated the authors of \cite{BB} to introduce the notions of ideally and projectively $\C$-closed semigroups. 

Let us recall that a {\em congruence} on a semigroup $X$ is an equivalence relation $\approx$ on $X$ such that for any elements $x\approx y$ of $X$ and any $a\in X$ we have $ax\approx ay$ and $xa\approx ya$. For any congruence $\approx$ on a semigroup $X$, the quotient set $X/_\approx$ has a unique semigroup structure such that the quotient map $X\to X/_\approx$ is a semigroup homomorphism. The semigroup $X/_\approx$ is called the {\em quotient semigroup} of $X$ by the congruence $\approx$~.
A subset $I$ of a semigroup $X$ is called an {\em ideal} in $X$ if $IX\cup XI\subseteq  I$. Every ideal $I\subseteq X$ determines the congruence $(I\times I)\cup \{(x,y)\in X\times X:x=y\}$ on $X\times X$. The quotient semigroup of $X$ by this congruence is denoted by $X/I$ and called the {\em quotient semigroup} of $X$ by the ideal $I$. If $I=\emptyset$, then the quotient semigroup $X/\emptyset$ can be identified with the semigroup $X$.

\begin{definition*}A semigroup $X$ is called
\begin{itemize}
\item {\em projectively $\C$-closed} if for any congruence $\approx$ on $X$ the quotient semigroup $X/_{\approx}$ is $\C$-closed;
\item {\em ideally $\C$-closed} if for any ideal $I\subseteq X$ the quotient semigroup $X/I$ is $\C$-closed.
\end{itemize}
\end{definition*}

It is easy to see that for every semigroup the following implications hold:
$$\mbox{absolutely $\C$-closed $\Ra$ projectively $\C$-closed $\Ra$ ideally $\C$-closed $\Ra$ $\C$-closed.}$$
It is easy to check that a semigroup $X$ is absolutely $\C$-closed if and only if for any congruence $\approx$ on $X$ the semigroup $X/_\approx$ is injectively $\C$-closed.

For a semigroup $X$, let $E(X)\defeq\{x\in X:xx=x\}$ be the set of idempotents of $X$. For an idempotent $e$ of a semigroup $X$, let $H_e$ be the maximal subgroup of $X$ that contains $e$. The union $H(X)=\bigcup_{e\in E(X)}H_e$ of all subgroups of $X$ is called the {\em Clifford part} of $S$.
A semigroup $X$ is called
\begin{itemize}
\item {\em Clifford}  if $X=H(X)$;
\item {\em Clifford+finite} if $X\setminus H(X)$ is finite;
\item {\em Clifford-finite} if the Clifford part $H(X)$ is finite.
\item {\em Clifford-singular} if there exists an infinite set $A\subseteq X\setminus H(X)$ such that $AA\subseteq H(X)$.
\end{itemize}

Ideally and projectively $\C$-closed commutative semigroups were characterized in  \cite{BB} as follows.

\begin{theorem}[Banakh--Bardyla]\label{t:mainP} Let $\C$ be a class of topological semigroups such that $\mathsf{T_{\!z}S}\subseteq\C\subseteq \mathsf{T_{\!1}S}$. For a commutative semigroup $X$ the following conditions are equivalent:
\begin{enumerate}
\item $X$ is projectively $\C$-closed;
\item $X$ is ideally $\C$-closed;
\item the semigroup $X$ is chain-finite, group-bounded and Clifford+finite.
\end{enumerate}
\end{theorem}


\begin{definition*}
Let $\C$ be a class of topological semigroups.
A semigroup $X$ is called
\begin{itemize}
\item {\em $\C$-discrete} (or else {\em $\C$-nontopologizable}) if for any injective homomorphism $h:X\to Y$ to a topological semigroup $Y\in\C$ the image $h[X]$ is a discrete subspace of $Y$;
\item {\em $\C$-topologizable} if $X$ is not $\C$-discrete;
\item {\em projectively $\C$-discrete} if for every homomorphism $h:X\to Y$ to a topological semigroup $Y\in\C$ the image $h[X]$ is a discrete subspace of $Y$.
\end{itemize}
\end{definition*}

The study of topologizable and nontopologizable semigroups is a classical topic in Topological Algebra that traces its history back to Markov's problem \cite{Markov} of topologizability of infinite groups, which was resolved in \cite{Shelah}, \cite{Hesse} and \cite{Ol} by constructing examples of nontopologizable infinite groups. For some other results on topologizability of semigroups, see \cite{BGP,BM,BPS,DS2,vD,DI,GLS,Gutik,KOO,Kotov,Taimanov}. 


For a semigroup $X$ let
$$Z(X)\defeq\{z\in X:\forall x\in X\;\;(xz=zx)\}$$be the {\em center} of $X$.
The first statement of following theorem is proved in Lemmas 5.1, 5.3, 5.4 of \cite{BB} and the second and third statements are taken from Theorem~1.7 in \cite{GCCS}.

\begin{theorem}[Banakh--Bardyla]\label{t:center} Let $X$ be a semigroup.
\begin{enumerate}
\item If $X$ is $\mathsf{T_{\!z}S}$-closed, then the semigroup $Z(X)$ is chain-finite, periodic and nonsingular.
\item  If $X$ is injectively $\mathsf{T_{\!z}S}$-closed or $\mathsf{T_{\!z}S}$-discrete, then $Z(X)$ is group-finite.
\item  If $X$ is ideally $\mathsf{T_{\!z}S}$-closed, then $Z(X)$ is group-bounded.
\end{enumerate}
\end{theorem}

Injectively $\C$-closed commutative unipotent semigroups were characterzed in \cite{CCUS} as follows.

\begin{theorem}[Banakh--Vovk]\label{t:unipotent} Let $\C$ be a class of topological semigroups such that $\mathsf{T_{\!z}S}\subseteq\C\subseteq \mathsf{T_{\!1}S}$.
 For a commutative unipotent semigroup $X$ the following conditions are equivalent:
\begin{enumerate}
\item $X$ is injectively $\C$-closed;
\item $X$ is $\C$-closed and group-finite;
\item $X$ is bounded, nonsingular and group-finite.
\end{enumerate}
\end{theorem}

The principal results of this paper are the following two theorems describing the center of an injectively (and absolutely) $\mathsf{T_{\!1}S}$-closed semigroup and also characterizing injectively (and absolutely) $\mathsf{T_{\!1}S}$-closed commutative semigroups.

\begin{theorem}\label{t:main-iT1} For a semigroup $X$ consider the following conditions:
\begin{enumerate}
\item $X$ is commutative, bounded, nonsingular and Clifford-finite;
\item $X$ is injectively $\mathsf{T_{\!1}S}$-closed;
\item $X$ is $\mathsf{T_{\!1}S}$-closed and $\mathsf{T_{\!1}S}$-discrete;
\item $X$ is $\mathsf{T_{\!z}S}$-closed and $\mathsf{T_{\!z}S}$-discrete;
\item $X$ is $\mathsf{T_{\!z}S}$-closed and $Z(X)$ is Clifford-finite;
\item $Z(X)$ is bounded, nonsingular and Clifford-finite;
\item $Z(X)$ is injectively $\mathsf{T_{\!1}S}$-closed.
\end{enumerate}
Then $(1)\Ra(2)\Leftrightarrow(3)\Ra(4)\Ra(5)\Ra(6)\Leftrightarrow(7)$.\\ If $X$ is commutative, then the conditions $(1)$--$(7)$ are equivalent.
\end{theorem}



\begin{theorem}\label{t:main-aT1} For a  semigroup $X$, consider the following conditions:
\begin{enumerate}
\item $X$ is finite;
\item $X$ is absolutely $\mathsf{T_{\!1}S}$-closed;
\item $X$ is projectively $\mathsf{T_{\!1}S}$-closed and projectively $\mathsf{T_{\!1}S}$-discrete;
\item $X$ is projectively $\mathsf{T_{\!1}S}$-closed and injectively $\mathsf{T_{\!1}S}$-closed;
\item $X$ is projectively $\mathsf{T_{\!1}S}$-closed and $\mathsf{T_{\!1}S}$-discrete;
\item $X$ is ideally $\mathsf{T_{\!z}S}$-closed and $\mathsf{T_{\!z}S}$-discrete;
\item $X$ is ideally $\mathsf{T_{\!z}S}$-closed and the semigroup $Z(X)$ is Clifford-finite;
\item $Z(X)$ is finite;
\item $Z(X)$ is absolutely $\mathsf{T_{\!1}S}$-closed.
\end{enumerate}
Then $(1)\Ra(2)\Leftrightarrow(3)\Ra(4)\Leftrightarrow(5)\Ra(6)\Ra(7)\Ra(8)\Leftrightarrow(9)$.\\ If $X$ is commutative, then the conditions $(1)$--$(9)$ are equivalent.
\end{theorem}

\begin{remark} The equivalences $(2)\Leftrightarrow(3)$ in Theorems~\ref{t:main-iT1} and \ref{t:main-aT1} were proved in Propositions 3.2 and 3.3 of \cite{BB2}. For viable semigroups the implication $(2)\Ra(8)$ of Theorem~\ref{t:main-aT1} was proved in Theorem 1.14 of \cite{ACS}.
\end{remark}

Theorems~\ref{t:main-iT1} and \ref{t:main-aT1} imply that the injective (and absolute) $\mathsf{T_{\!1}S}$-closedness is preserved by subsemigroups of commutative semigroups.

\begin{corollary}\label{c:her} Any subsemgroup of an injectively (and absolutely) $\mathsf{T_{\!1}S}$-closed commutative semigroup is injectively (and absolutely) $\C$-closed.
\end{corollary}

\begin{remark} Corollary~\ref{c:her} is specific for commutative semigroups and does not generalize to noncommutative groups: by Theorem 1.10 in \cite{BB2},  every countable bounded group $G$ without elements of order 2 is a subgroup of an absolutely $\mathsf{T_{\!1}S}$-closed countable simple bounded group $X$. If the group $G$ has infinite center, then $G$ is not injectively $\mathsf{T_{\!1}S}$-closed by Theorem~\ref{t:main-iT1}. On the other hand, $G$ is a subgroup of the absolutely $\mathsf{T_{\!1}S}$-closed group $X$. This example also shows that the equivalences $(1)\Leftrightarrow(2)$ in Theorems~\ref{t:main-iT1} and \ref{t:main-aT1} do not hold for non-commutative groups.
\end{remark}

Theorems~\ref{t:main-iT1} and \ref{t:main-aT1} will be proved in Sections~\ref{s:main-iT1} and \ref{s:main-aT1}. The main instrument in the proof of Theorem~\ref{t:main-iT1} is Theorem~\ref{t:topologizable} on $\mathsf{T_{\!z}S}$-topologizability of semigroups $X$ whose central semilattice $EZ(X)=E(X)\cap Z(X)$ is chain-finite and infinite. This topologizability theorem is proved using semigroup topologies, generated by remote bases. The corresponding technique is elaborated in Sections~\ref{s:Lambda}--\ref{s:topology}. The obtained topologizability results have an independent value and are essentially used in the paper \cite{ICVS}, containing the following characterization of injectively $\C$-closed commutative semigroups.

\begin{theorem}[Banakh] Let $\C$ be a class of topological semigroups such that $\mathsf{T_{\!z}S}\subseteq\C\subseteq\mathsf{T_{\!2}S}$. A commutative semigroup $X$ is injectively $\C$-closed if and only if $X$ is chain-finite, group-finite, bounded, nonsingular and not Clifford-singular.
\end{theorem}

\section{Preliminaries}\label{s:prelim}

We denote by $\w$ the set of finite ordinals, by $\IN\defeq\w\setminus\{0\}$ the set of positive integer numbers. 
For a set $X$ we denote by $[X]^{<\w}$  the family of all finite subsets of $X$.

A {\em poset} is a set $X$ endowed with a partial order $\le$. For an element $a$ of a poset $X$, let  
$${\downarrow}a\defeq\{x\in X :x\le a\}\quad\mbox{and}\quad{\uparrow}a\defeq\{x\in X:a\le x\}$$
be the {\em lower} and {\em upper sets} of $a$ in $X$, respectively. 

For a subset $A$ of a poset $X$, let 
$${\downarrow}A\defeq\bigcup_{a\in A}{\downarrow}a\quad\mbox{and}\quad
{\uparrow}A\defeq\bigcup_{a\in A}{\uparrow}a
$$be the {\em lower} and {\em upper sets} of $A$ in the poset $X$.

For two elements $x,y$ of a poset $X$ we write $x<y$ if $x\le y$ and $x\ne y$.

A subset $A$ of a poset $X$ is called a {\em chain} if for any $x,y\in A$ either $x\le y$ or $y\le x$.

A poset $X$ is called
\begin{itemize}
\item {\em chain-finite} if each chain in $X$ is finite;
\item {\em well-founded} if every nonempty set $A\subseteq X$ contains an element $a\in A$ such that $A\cap{\downarrow}a=\{a\}$.
\end{itemize}
It is easy to see that each chain-finite poset is well-founded.

Let $X$ be a semigroup and $E(X)\defeq\{x\in X:xx=x\}$ be the set of idempotents of $X$. We shall consider $E(X)$ as a poset endowed with the {\em natural partial order} $\le$ defined by $x\le y$ iff $xy=yx=x$. Observe that for a (commutative) semigroup $X$ the poset $E(X)$ is chain-finite if (and only if) the semigroup $X$ is chain-finite.

For any infinite set $X$ endowed with the left zero multiplication $xy=x$, the poset $E(X)=X$ is chain-finite but the semigroup $X$ is not chain-finite.



An element $z$ of a semigroup $X$ is called {\em central} if $z\in Z(X)\defeq\{z\in X:\forall x\in X\;\;(zx=xz)\}$. The intersection $$EZ(X)\defeq E(X)\cap Z(X)=E(Z(X))$$is called the {\em central semilattice} of $X$.

For an element $a$ of a semigroup $X$, the set
$$H_a\defeq\{x\in X:(xX^1=aX^1)\;\wedge\;(X^1x=X^1a)\}$$
is called the {\em $\mathcal H$-class} of $a$.
Here $X^1\defeq X\cup\{1\}$ where $1$ is an element such that $1x=x=x1$ for all $x\in X^1$. By Corollary 2.2.6 \cite{Howie}, for every idempotent $e\in E(X)$ its $\mathcal H$-class $H_e$ coincides with the maximal subgroup of $X$, containing the idempotent $e$.
The union
$$H(X)\defeq \bigcup_{e\in E(X)}H_e$$is called the {\em Clifford part} of $X$.
The Clifford part is not necessarily a subsemigroup of $X$.

On the other hand,
the {\em central Clifford part}
$$H_Z(X)\defeq \bigcup_{e\in EZ(X)}H_e$$is a subsemigroup of $X$.

\begin{lemma}\label{l:HZ} For every semigroup $X$ the central Clifford part $H_Z(X)$ is a subsemigroup of $X$.
\end{lemma}

\begin{proof} Given any $x,y\in H_Z(X)$, find central idempotents $e,f\in EZ(X)$ such that $x\in H_e$ and $y\in H_f$. Since the idempotents $e,f$ are central, the product $fe$ is a central idempotent in $X$. Observe that
$$xyX^1=xfX^1=fxX^1=feX^1\quad\mbox{and}\quad X^1xy=X^1ey=X^1ye=X^1fe,$$which means that $xy\in H_{fe}\subseteq H_Z(X)$.
\end{proof}

For any element $x\in H(X)$, there exists a unique element $x^{-1}\in H(X)$ such that $$xx^{-1}x=x,\quad x^{-1}xx^{-1}=x^{-1},\quad\mbox{and}\quad xx^{-1}=x^{-1}x.$$

\begin{lemma}\label{l:ZH} Let $X$ be a semigroup. If for some $e\in E(X)$ the intersection $H_e\cap Z(X)$ is not empty, then $Z(X)\cap H_e$ is a subgroup of  $H_e$ and $e\in Z(X)$.
\end{lemma}

\begin{proof} It is clear that $Z(X)\cap H_e$ is a subsemigroup of $H_e$. It remains to prove that $z^{-1}\in Z(X)\cap H_e$ for any $z\in Z(X)\cap H_e$. Given any $z\in Z(X)\cap H_e$ and $x\in X$, we have $zx=xz$ and hence $exz=ezx=zx=xz$ and $zx=xz=xze=zxe$. Multiplying the equalities $zx=zxe$ and $exz=xz$ by $z^{-1}$, we obtain $ex=z^{-1}zx=z^{-1}zxe=exe=exzz^{-1}=xzz^{-1}=xe$ and hence $e\in Z(X)$.

Multiplying the equality $xz=zx$ by $z^{-1}$ from the left, we obtain $z^{-1}xz=z^{-1}zx=ex=xe=xz^{-1}z$. The equality $z^{-1}xz=xz^{-1}z$ implies
$$z^{-1}x=z^{-1}ex=z^{-1}xe=z^{-1}xzz^{-1}=xz^{-1}zz^{-1}=xz^{-1},$$which means that $z^{-1}\in Z(X)$.
\end{proof}

\begin{lemma}\label{l:inverse} Let $X$ be a semigroup and $x,y\in H(X)$. If $xy=yx$, then $xy\in H(X)$ and $(xy)^{-1}=x^{-1}y^{-1}=y^{-1}x^{-1}$.
\end{lemma}

\begin{proof} Consider the idempotents $e=xx^{-1}=x^{-1}x$ and $f=yy^{-1}=y^{-1}y$ and observe that
\begin{multline*}ey=x^{-1}xy=x^{-1}yx=x^{-1}yxx^{-1}x=x^{-1}yxe=x^{-1}xye=eye\\=eyxx^{-1}=exyx^{-1}=xyx^{-1}=yxx^{-1}=ye.
\end{multline*}
By analogy we can prove that $xf=fx$. Next, observe that
$$ef=eyy^{-1}=yey^{-1}=fyey^{-1}=feyy^{-1}=fef=
y^{-1}yef=y^{-1}eyf=y^{-1}ey=y^{-1}ye=fe.$$
Then for the idempotent $u=ef=fe$ we have
$xyX^1=xfX^{1}=fxX^1=feX^1=uX^1$ and $X^1xy=X^1ey=X^1ye=X^1fe=X^1u$,
 which means that $xy\in H_u\subseteq H(X)$.
Observe that $$x^{-1}f=x^{-1}ef=x^{-1}fe=x^{-1}fxx^{-1}=x^{-1}xfx^{-1}=efx^{-1}=fex^{-1}=fx^{-1}.$$
By analogy we can prove that $y^{-1}e=ey^{-1}$.
Then $x^{-1}y^{-1}X^1=x^{-1}fX^1=fx^{-1}X^1=feX^1=uX^1$ and $X^1x^{-1}y^{-1}=X^1ey^{-1}=X^1y^{-1}e=X^1fe=X^1u$, which means that $x^{-1}y^{-1}\in H_u$. By analogy we can prove that $y^{-1}x^{-1}\in H_u$. It follows from $xyy^{-1}x^{-1}=xfx^{-1}=fxx^{-1}=fe=u$ that $y^{-1}x^{-1}=(xy)^{-1}$. Also
$xyx^{-1}y^{-1}=yxx^{-1}y^{-1}=yey^{-1}=eyy^{-1}=ef=u$ implies that $x^{-1}y^{-1}=(xy)^{-1}=y^{-1}x^{-1}$.
\end{proof}

For a subset $A$ of a semigroup $X$ and a positive integer number $n$, let
$$\korin{n}{A}\defeq\{x\in X:x^n\in A\}\quad\mbox{and}\quad\korin{\infty}{A}\defeq\bigcup_{n\in\IN}\korin{n}{A}=\{x\in X:A\cap x^\IN\ne\emptyset\},$$ where $$x^\IN=\{x^k:k\in\IN\}$$is the {\em monogenic semigroup} generated by $x$.

 For a point $a\in X$, the set $\korin{\infty}{\{a\}}$ will be denoted by $\korin{\infty}{\,a}$. The sets $\korin{\infty}{E(X)}$ and $\korin{\infty}{H(X)}$ are called the {\em periodic part} and {\em eventually Clifford part} of $X$, respectively.

A semigroup $X$ is called {\em eventually Clifford} if $X=\korin{\infty}{H(X)}$.
It is clear that each periodic semigroup is eventually Clifford (but not vice versa).

The following lemma is proved in \cite[3.1]{BB}.

\begin{lemma}\label{l:C-ideal} For any idempotent $e$ of a semigroup we have $(\!\korin{\infty}{H_e}\cdot H_{e})\cup(H_{e}\cdot \korin{\infty}{H_e}\,)\subseteq H_{e}.$
\end{lemma}

\begin{lemma}\label{l:pi-well-defined} Let $x$ be an element of a semigroup $X$ such that $x^n\in H_e$ for some $n\in\IN$ and $e\in E(X)$. Then $x^m\in H_e$ for all $m\ge n$.
\end{lemma}

\begin{proof} To derive a contradiction, assume that $x^m\notin H_e$ for some $m\ge n$. We can assume that $m$ is the smallest number such that $m\ge n$ and  $x^m\notin H_e$. It follows from $x^n\in H_e$ and $x^m\notin H_e$ that $m>n>1$ and hence $m-2\in\IN$. The minimality of $m$ ensures that $x^{m-1}\in H_e$. Observe that $x^{m}X^1\subseteq x^{m-1}X^1=ex^{m-1}X^1\subseteq eX^1$ and $$eX^1=x^{2(m-1)}(x^{2(m-1)})^{-1}X^1\subseteq x^{2(m-1)}X^1=x^{m}x^{m-2}X^1\subseteq x^{m}X^1.$$Therefore, $x^mX^1=eX^1$. By analogy one can prove that $X^1x^{m}=X^1e$. Therefore, $x^{m}\in H_e$, which contradicts the choice of $m$.
\end{proof}

For a semigroup $X$, let $\pi:\korin{\infty}{H(X)}\to E(X)$ be the map assigning to each $x\in X$ a unique idempotent $\pi(x)$ such that $x^\IN\cap H_{\pi(x)}\ne\emptyset$. Lemma~\ref{l:pi-well-defined} ensures that the map $\pi$ is well-defined.

\begin{lemma}\label{l:pi-homo} If $X$ is a commutative semigroup, then $\korin{\infty}{H(X)}$ is a subsemigroup of $X$ and $\pi:\korin{\infty}{H(X)}\to E(X)$ is a homomorphism.
\end{lemma}

\begin{proof} Given any $x,y\in \korin{\infty}{H(X)}$, find $n\in\IN$ such that $x^n\in H_{\pi(x)}$ and $y^n\in H_{\pi(y)}$. By (the proof of) Lemma~\ref{l:inverse}, $(xy)^n=x^ny^n\in H_{\pi(x)\pi(y)}\subseteq H(X)$ and hence $xy\in\korin{\infty}{H(X)}$, and $\pi(xy)=\pi(x)\pi(y)$, which means that $\korin{\infty}{H(X)}$ is a subsemigroup of $X$ and $\pi$ is a homomorphism.
\end{proof}

\section{Shifting sets in semigroups}\label{s:Lambda}

In this section we describe the operation of shifting subsets in a semigroup, which allows to transport subsets of a semigroup from one place to another.

Let $X$ be a semigroup. Given two elements $e,b\in X$, consider the set
$$\tfrac{b}e\defeq\{x\in X:xe=b\}$$which can be thought as the set of all left shifts that  move $e$ to $b$. If the set $\frac be$ is not empty, then $\frac be\cdot e=\{b\}$ and for any subset $U\subseteq X$ containing $e$, the set $\frac{b}e\cdot U$ contains $b$.

The assignment $$U\mapsto \Lambdae(b;U)\defeq\{b\}\cup\big(\tfrac{b}e\cdot U\big)$$ will be referred to as the {\em  $e$-to-$b$ shift} of $U$. 

Let us describe some properties of the $e$-to-$b$ shifts.

\begin{lemma}\label{l:Lamb} Let $e$ be an idempotent of a semigroup $X$. For any elements $a,b\in X$ and subsets $U,V,W\subseteq X$ the following statements hold:
\begin{enumerate}
\item If $V\subseteq W$, then $\Lambdae(b;V)\subseteq \Lambdae(b;W)$.
\item $\Lambdae(b;\frac ee)\subseteq\tfrac{be}{e}$.
\item If $b\ne be$, then $\Lambdae(b;\frac ee)=\{b\}$.
\item If $a\in\Lambdae(b;\frac ee)\setminus\{b\}$, then $\Lambdae(a;\frac ee)=\{a\}$.
\item If $a\ne b$ and $\Lambdae(a;\frac ee)\cap\Lambdae(b;\frac ee)\ne \emptyset$, then either $\Lambdae(a;\frac ee)=\{a\}$ or $\Lambdae(b;\frac ee)=\{b\}$.
\item If $V\subseteq W$, then $a\cdot\Lambdae(b;V)\subseteq \Lambdae(ab;W)$.
\item If $Ub\subseteq bW$ and $be=eb$, then $\Lambdae(a;U)\cdot b\subseteq\Lambdae(ab;W)$;
\item If $e\in Z(X)$, $V\subseteq W$, $Ub\subseteq bW$ and $\forall y\in \frac{b}e\;(UyV\subseteq yW)$, then
 $\Lambdae(a;U)\cdot\Lambdae(b;V)\subseteq \Lambdae(ab;W)$.
\end{enumerate}
\end{lemma}

\begin{proof} 1. If $V\subseteq W$, then $\Lambdae(b;V)=\{b\}\cup\big(\frac be\cdot V\big)\subseteq\{b\}\cup\big(\frac be\cdot W\big)=\Lambdae(b;W)$.
\smallskip

2.  Fix any $x\in\Lambdae(b;\frac ee)$.  If $x=b$, then $xe=be$ and hence $x\in\frac{be}e$.  If $x\ne b$, then $x=us$ for some $u\in\frac be$ and $s\in \frac ee$. Then $xe=use=ue=uee=be$ and again $x\in\frac{be}e$.
\smallskip

3. If $b\ne be$, then $\frac be=\emptyset$ and hence $\Lambdae(b;V)=\{b\}$.
\smallskip

4. If $a\in\Lambda(b,\frac ee)\setminus\{b\}\subseteq\frac be\cdot\frac ee$, then $a=b'v$ for some $b'\in\frac be$ and $v\in\frac ee$. Then $ae=b've=b'e=b\ne a$ and $\Lambda(a;\frac ee)=\{a\}$ by Lemma~\ref{l:Lamb}(3).
\smallskip

5. Assume that $a\ne b$ and $\Lambdae(a;\frac ee)\cap\Lambdae(b;\frac ee)\ne \emptyset$. If $a\ne ae$ or $b\ne be$, then $\Lambdae(a;\frac ee)=\{a\}$ or $\Lambdae(b;\frac ee)=\{b\}$ by Lemma~\ref{l:Lamb}(3). So, we assume that $a=ae$ and $b=be$. Take any element $x\in\Lambdae(a;\frac ee)\cap\Lambdae(b;\frac ee)$ and observe that $x\in \frac{ae}e\cap\frac{be}e$, by Lemma~\ref{l:Lamb}(2). Then $a=ae=xe=be=b$, which contradicts the choice of $a,b$.
\smallskip

6--8. Take any elements $x\in\Lambdae(a;U)$ and $y\in\Lambdae(b;V)$.

If $x=a$ and $y=b$, then $xy=ab\in \Lambdae(ab;W)$.

If $V\subseteq W$, $x=a$ and $y\ne b$, then $y=b'v$ for some $b'\in\frac be$ and $v\in V\subseteq W$. It follows from $b'\in\frac be$ that $b'e=b$ and $ab'e=ab$ and finally $ab'\in\frac{ab}e$. Now we see that $xy=ab'v\in \tfrac{ab}e\cdot V\subseteq\Lambdae(ab,V)\subseteq \Lambdae(ab;W)$
and hence $a\cdot\Lambdae(b;V)\subseteq\Lambdae(ab;W)$.

If $Ub\subseteq bW$, $x\ne a$ and $y=b$, then $x=a'u$ for some $a'\in\frac ae$ and $u\in U$. It follows from $ub\in Ub\subseteq bW$ that $ub=bw$ for some $w\in W$. If $be=eb$, then $a'be=a'eb=ab$ and hence $a'b\in\frac{ab}e$. Then
$xy=a'ub=a'bw\in\tfrac{ab}e\cdot W\subseteq \Lambdae(ab;W)$
and hence $\Lambdae(a;U)\cdot b\subseteq\Lambdae(ab;W)$.

Finally assume that $x\ne a$, $y\ne b$, and $UcV\subseteq cW$ for every $c\in \frac be$. In this case $x=a'u$ and $y=b'v$ for some $a'\in \frac ae$, $b'\in\frac be$, $u\in U$ and $v\in V$. If $e\in Z(X)$, then $a'b'e=(a'e)(b'e)=ab$, which implies $a'b'\in\frac {ab}e$ and finally $xy=a'ub'v\in a'b'W\subseteq\tfrac{ab}e\cdot W\subseteq \Lambdae(ab;W).$
\end{proof}

\section{Topologies generated by remote bases on semigroups}

In this section we introduce the notion of a remote base on a semigroup $X$ and prove that it generates a $T_0$ semigroup topology on $X$. Also we provide a condition ensuring that this topology is zero-dimensional. We recall that a topological space $X$ satisfies the separation axiom $T_0$ (or else $X$ is a {\em $T_0$-space}) if for any distinct points $x,y\in X$ there exists an open set $U\subseteq X$ such that $U\cap\{x,y\}$ is a singleton.

\begin{definition}\label{d:remote}  Let $X$ be a semigroup and $e$ be a central idempotent in $X$. An {\em $e$-remote base} on $X$ is a function $\Phi=(\Phi_x)_{x\in X}$ assigning to each $x\in X$ a family $\Phi_x$ of subsets of $X$ satisfying the following conditions:
\begin{enumerate}
\item $\forall x\in X\;\forall A,B\in\Phi_x\;\exists C\in\Phi_x\;\;\big(C\subseteq A\cap B\subseteq\frac ee\big)$;
\item $\forall x,y\in X\;\forall W\in\Phi_{xy}\;\;\exists U\in\Phi_x\;\exists V\in\Phi_y\;\;\big(V\subseteq W\;\wedge\; Uy\subseteq yW\;\wedge\;\forall b\in\frac ye \;(UbV\subseteq bW)\big)$.
\end{enumerate}
\smallskip

Given an $e$-remote base $\Phi=(\Phi_x)_{x\in X}$, let  $\Tau_\Phi$ be the topology on $X$, consisting of all sets $W\subseteq X$ such that for every $x\in W$ there exist a set $U\in\Phi_x$ such that $x\in\Lambdae(x;U)\subseteq W$. The topology $\Tau_\Phi$ will be referred to as {\em the topology generated by the $e$-remote base $\Phi$}.
\end{definition}

Lemma~\ref{l:Lamb}(1,4) implies the following  lemma.

\begin{lemma}\label{l:base-e} Let $e$ be a central idempotent in a semigroup $X$ and $\Phi$ be an $e$-remote base. For every $x\in X$ the family $$\mathcal B_x\defeq\big\{\Lambdae(x;V):V\in\Phi_x\big\}$$ is a neighborhood base of the topology $\Tau_\Phi$ at $x$.
\end{lemma}




Now we define a condition on an $e$-remote base $\Phi$ implying the  zero-dimensionality of the topology $\Tau_\Phi$.

\begin{definition}\label{d:regular} Let $X$ be a semigroup and $e$ be a central idempotent in $X$. An $e$-remote base $\Phi$  is defined to be {\em regular} if for any element $b\in X$ with $b\ne be$, there exists a set $V\in \Phi_{be}$ such that $b\notin \frac{be}{e}\cdot V$.
\end{definition}

\begin{theorem}\label{t:TS} Let $X$ be a  semigroup, $e$ be a central idempotent in $X$ and $\Phi=(\Phi_x)_{x\in X}$ be an $e$-remote base. Then
\begin{enumerate}
\item $(X,\Tau_{\Phi})$ is a topological semigroup;
\item $(X,\Tau_{\Phi})$ is a $T_0$ topological space with discrete subspace of non-isolated points.
\item If the $e$-remote base $\Phi$ is regular, then for every point $b\in X$, any subset $B\subseteq \Lambdae(b;\frac ee)$ containing $b$ is closed in the topology $\Tau_{\Phi}$.
\item If the $e$-remote base $\Phi$ is regular, then every topology $\tau$ on $X$ with $\tau_\Phi\subseteq\tau$ is Hausdorff and zero-dimensional.
\end{enumerate}
\end{theorem}

\begin{proof} By Lemma~\ref{l:base-e}, the family
$$\mathcal B\defeq\big\{\Lambdae(x;V):x\in X,\;V\in\Phi_x\big\}$$
is a base of the topology $\Tau_\Phi$.
\smallskip

1. To see that $(X,\Tau_\Phi)$ is a topological semigroup, take any elements $a,b\in X$ and a neighborhood $O_{ab}\in\Tau_\Phi$ of $ab$. Find a set $W\in\Phi_{ab}$ such that $\Lambdae(ab;W)\subseteq O_{ab}$. By Definition~\ref{d:remote}, there exist sets $U\in\Phi_a$ and $V\in\Phi_b$ such that $V\subseteq W$, $Ub\subseteq bW$ and $\forall y\in\frac be\;\;(UyV\subseteq yW)$. By Lemmas~\ref{l:base-e} and \ref{l:Lamb}(8), $\Lambdae(a;U)$ and $\Lambdae(b;V)$ are $\Tau_\Phi$-open sets such that $a\in\Lambdae(a;U)$, $b\in\Lambdae(b;V)$ and $$\Lambdae(a;U)\cdot\Lambdae(b;V)\subseteq\Lambdae(ab;W)\subseteq O_{ab}.$$
So, $(X,\Tau_\Phi)$ is a topological semigroup.
\smallskip

2. Lemma~\ref{l:Lamb}(5) ensures that the topology $\Tau_\Phi$  satisfies the separation axiom $T_0$.  To see that the subspace $X'$ of non-isolated points of $(X,\Tau_\Phi)$ is discrete, take any point $x'\in X'$ and consider the neighborhood $\Lambdae(x',\frac ee)$ of $x'$ in $(X,\Tau_\Phi)$. By Lemma~\ref{l:Lamb}(4), every point $x\in \Lambda(x',\frac ee)\setminus\{x'\}$ is isolated in $(X,\Tau_\Phi)$, which implies that $\{x'\}=X'\cap \Lambdae(x',\frac ee)$ and hence $x'$ is an isolated point of the set $X'$, so the subspace $X'$ of $(X,\Tau_\Phi)$ is discrete.
\smallskip

3. Assume that the $e$-remote base $\Phi$ is regular. Take any $b\in X$ and any set $B\subseteq \Lambdae(b;\frac ee)$ such that $b\in B$.  Assuming that $B$ is not closed in the topology $\Tau_{\Phi}$, we can find an element $a\notin B$ such that for any set $U\in\Phi_a$ we have $$\emptyset\ne \Lambdae(a;U)\cap B\subseteq \Lambdae(a;\tfrac ee)\cap\Lambdae(b;\tfrac ee)$$ and hence $\{a\}\ne\Lambdae(a;U)\subseteq\Lambdae(a;\tfrac ee)$. Applying Lemma~\ref{l:Lamb}(4), we conclude that $b\in B\subseteq\Lambdae(b;\frac ee)=\{b\}$. Then $b\in \Lambdae(a;U)$ and hence $b=a'u$ for some $a'\in\frac{a}e$ and $u\in U\subseteq \frac ee$. It follows that $be=(a'u)e=a'(ue)=a'e=a\ne b$.  Since the remote base $\Phi$ is regular, there exists a set $U\in\Phi_{be}=\Phi_a$ such that $b\notin \frac{be}{e}\cdot U$ and hence $b\notin \Lambdae(be;U)=\Lambdae(a;U)$, which contradicts $a\in\overline{B}\subseteq \overline{\Lambdae(b;\frac ee)}=\overline{\{b\}}$.
\smallskip

4. Assume that the $e$-remote base $\Phi$ is regular and take any topology $\tau$ on $X$ with $\Tau_\Phi\subseteq\tau$. Since $(X,\Tau_\Phi)$ is a $T_0$-space, the topological space $(X,\tau)$ is a $T_0$-space, too. To see that $(X,\Tau)$ is zero-dimensional, take any open set $U\in\tau$ and any point $x\in U$. Since $\Tau_\Phi\subseteq \tau$, the set $O_x\defeq U\cap \Lambda(x;\frac ee)$ is a neighborhood of $x$ in the topology $\tau$. By Theorem~\ref{t:TS}(3), the set $O_x$ is closed in the topology $\Tau_\Phi$ and hence is closed in the topology $\tau$. Therefore, the topological space $(X,\tau)$ is zero-dimensional and being a $T_0$-space,  is Hausdorff.
\end{proof}

Now we present two easy-to-apply conditions of regularity of a remote base. We recall that $EZ(X)=E(Z(X))=E(X)\cap Z(X)$ is the {\em central semilattice} of a semigroup $X$ and $H_Z(X)\defeq\bigcup_{e\in EZ(X)}H_e$ is the {\em central Clifford part} of $X$. By $\pi:\korin{\infty}{H(X)}\to E(X)$ we denote the map assigning to each $x\in\korin{\infty}{H(X)}$ the unique idempotent $\pi(x)$ such that $x^n\in H_{\pi(x)}$ for some $n\in\IN$.

\begin{proposition}\label{p:regular} Let $X$ be a semigroup  and $\Phi$ be an $e$-remote base at a central idempotent $e$ of $X$. Assume  that the semilattice $EZ(X)$ is well-founded, and for every element $b\ne be$ in $X$ and every finite set $F\subseteq EZ(X)\setminus {\downarrow}e$, there exists a set $V\in\Phi_{be}$ such that $V\subseteq H_Z(X)\setminus\pi^{-1}[{\uparrow}F]$. Then the remote base $\Phi$ is regular and the topological semigroup $(X,\Tau_\Phi)$ is  zero-dimensional.
\end{proposition}

\begin{proof} To prove that the remote base $\Phi$ is regular, take any element $b\ne be$ in $X$. 
 Observe that  $L=\{x\in EZ(X):bx=b\}$ is a subsemigroup of $EZ(X)$. If $L$ is empty, then put $F=\emptyset$. If $L\ne\emptyset$, then let $F=\{s\}$ where $s$ is the smallest element of the poset $L$, which exists by the well-foundedness of the semilattice $EZ(X)$.

We claim that $F\subseteq EZ(X)\setminus{\downarrow}e$. In the opposite case $F\ne\emptyset$, $L\ne \emptyset$ and $se=es=s$ and  $b=bs=bse=be$, which contradicts our assumption. So, $F\subseteq EZ(X)\setminus{\downarrow}e$. By the assumption of the proposition, there exists a set $V\in\Phi_{be}$ such that $V\subseteq H_Z(X)\setminus\pi^{-1}[{\uparrow}F]$. 
We claim that $b\notin \frac {be}e\cdot V$. In the opposite case we can find elements $b'\in \frac {be}e$ and $v\in V\subseteq H_Z(X)$ such that $b=b'v$. Since $v\in H_Z(X)$, there exists an idempotent $c\in EZ(X)$ such that $v\in H_c$. Then $b=b'v=b'vc=bc$ and hence $c\in L\ne\emptyset$ and $\pi(v)=c\in{\uparrow}s={\uparrow}F$ But $\pi(v)\in{\uparrow}F$ contradicts the choice of $v\in V\subseteq H_Z(X)\setminus\pi^{-1}[{\uparrow}F]$. This contradiction shows that $b\notin \frac{be}e\cdot V$, which completes the proof of the regularity of the remote base $\Phi$. By Theorem~\ref{t:TS}, the topological semigroup $(X,\Tau_\Phi)$ is zero-dimensional.
\end{proof}

A less trivial criterion of regularity of a remote base is supplied by the following proposition. We recall that a semigroup $X$ is {\em eventually Clifford} if $X=\korin{\infty}{H(X)}$.

\begin{proposition}\label{p:regular-powers} Let $X$ be a nonsingular eventually Clifford semigroup such that the poset $E(X)$ is well-founded. Let $\Phi$ be an $e$-remote base for $X$ at a central idempotent $e\in E(X)$ such that for every $n\in\IN$, finite set $F\subseteq  E(X)\setminus {\downarrow}e$,  and every $b\in X$ with $b\ne be$, there exists a set $V\in\Phi_{be}$ such that $V\subseteq\{z^m:z\in  Z(X)\cap\frac ee \setminus\pi^{-1}[{\uparrow}F],\;m\ge n\}$. Then the $e$-remote base $\Phi$ is regular and the topological semigroup $(X,\Tau_\Phi)$ is Hausdorff and zero-dimensional.
\end{proposition}

\begin{proof} To prove that the $e$-remote base $\Phi$ is regular, take any element $b\in X$ such that $b\ne be$. We need to find a set $V\in\Phi_{be}$ such that $b\notin \frac{be}{e}\cdot V$. By Lemma~\ref{l:powers} below, there exist $n\in\IN$ and a finite set $F\subseteq E(X)\setminus{\downarrow}e$ such that $b\notin\{az^m:a\in X,\;z\in Z(X)\cap\frac ee \setminus\pi[{\uparrow}F],\;m\ge n\}$. By our assumption, there exists a set $V\in\Phi_{be}$ such that $V\subseteq\{z^m:z\in Z(X)\cap\frac ee \setminus\pi[{\uparrow}F],\;m\ge n\}$.

We claim that $b\notin\frac {be}e\cdot V$. Indeed, assuming that $b\in\frac{be}e\cdot V$, we can find elements $a\in\frac{be}e$ and $v\in V$ such that $b=av$. By the choice of $V$, there exists $z\in Z(X)\cap\frac ee\setminus\pi[{\uparrow}F]$ such that $v=z^m$ for some $m\ge n$. Then $b=az^m$, which contradicts the choice of $n$ and $u$. To complete the proof of Proposition~\ref{p:regular-powers}, it remains to prove the following lemma.
\end{proof}

\begin{lemma}\label{l:powers} Let $X$ be a nonsingular eventually Clifford semigroup such that the poset $E(X)$ is well-founded. For any $e\in EZ(X)$ and $b\in X$ with $b\ne be$, there exist $n\in\IN$ and $F\in[E(X)\setminus{\downarrow}e]^{<\w}$ such that $ue\ne u$ and $b\notin \{az^m:a\in X,\;z\in  Z(X)\cap\frac ee \setminus\pi^{-1}[{\uparrow}F],\;m\ge n\}$.
\end{lemma}

\begin{proof} 
Fix $e\in EZ(X)$ and $b\in X$ with $b\ne be$. If $\pi(b)e\ne\pi(b)$, then the number $n=1$ and the set $F=\{\pi(b)\}$ have the required properties. Indeed, assume that $b=az^m$ for some $a\in X$, $m\in\IN$ and $z\in  Z(X)\cap\frac ee \setminus\pi^{-1}[{\uparrow} F]$.
Since $X$ is eventually Clifford, there exists $k\in\IN$ such that $b^k\in H_{\pi(b)}$, $a^k\in H_{\pi(a)}$ and $z^{mk}\in H_{\pi(z)}$. Let $(b^k)^{-1}$, $(a^k)^{-1}$ and $(z^{mk})^{-1}$ be the inverse elements to $b^k,a^k,z^{mk}$ in the groups $H_{\pi(b)}$, $H_{\pi(a)}$ and $H_{\pi(z)}$, respectively. By Lemma~\ref{l:inverse}, the equality $b^k=(az^m)^k=a^kz^{mk}$ implies $(b^k)^{-1}=(a^k)^{-1}(z^{mk})^{-1}$. Multiplying the last equality by $b^k=a^kz^{mk}$, we obtain
$$\pi(b)=b^k(b^k)^{-1}=a^kz^{mk}(a^k)^{-1}(z^{mk})^{-1}=a^k(a^k)^{-1}z^{mk}(z^{mk})^{-1}=\pi(a)\pi(z)=\pi(z)\pi(a),$$which contradicts the choice of $z\notin\pi^{-1}[{\uparrow} F]$.
\smallskip

So, we assume that $\pi(b)e=\pi(b)$.
Since $X$ is eventually Clifford, there exists $n\in\IN$ such that $b^n\in H_{\pi(b)}$. Then $b^n=b^n\pi(b)=b^n(\pi(b)e)=(b^n\pi(b))e=b^ne$. Let $q\in\IN$ be the unique number $q\in\IN$ such that $b^q\ne b^qe$ and $b^{q+1}=b^{q+1}e$. 

Let $X^1=X\cup \{1\}$ where $1\notin X$ be an element such that $1x=x=x1$ for all $x\in X^1$. In the set $E(X^1)$ consider the subset $$L\defeq\{u\in E(X^1):\exists \zeta\in Z(X^1)\quad (b^q\zeta u=b^q\zeta\ne b^q\zeta e)\}.$$ The set $L$ contains $1$ and hence is not empty. Since the poset $E(X)$ is well-founded, there exists $u\in L$ such that $L\cap{\downarrow}u=\{u\}$. Let $F\defeq X\cap\{u\}$.

It remains to prove that there exists $n\in\IN$ such that $$b\notin\{az^m:a\in X,\;z\in  Z(X)\cap\tfrac ee \setminus\pi^{-1}[{\uparrow}F],\;m\ge n\}.$$ To derive a contradiction, assume that for every $n\in \IN$ there exist $a\in  X$ and $z\in  Z(X)\cap\tfrac ee \setminus\pi^{-1}[{\uparrow}F]$ such that $b=az^m$ for some $m\ge n$. 

By the definition of the set $L\ni u$, there exists $\zeta\in Z(X^1)$ such that $b^q\zeta u=b^q\zeta\ne b^q\zeta e$ and hence $eu\ne u$.  Let $$c\defeq b^q\zeta$$ and observe that
$$cu=c\ne ce.$$

\begin{claim}\label{cl:b2q} $c^2e=c^2$.
\end{claim}

\begin{proof} If $q=1$, then $b^{2q}e=b^{q+1}e=b^{q+1}=b^{2q}$ by the choice of $q$. If $q>1$, then $b^{2q}e=b^{q-1}b^{q+1}e=b^{q-1}b^{q+1}=b^{2q}$. In both cases we have $b^{2q}e=b^{2q}$ which implies $$c^2e=(b^q\zeta)^2e=b^{2q}\zeta^2e =b^{2q}e\zeta^2=b^{2q}\zeta^2=(b^q\zeta)^2=c^2.$$
\end{proof}

\begin{claim}\label{cl:exist-azn} For every $n\in \IN$, there exist $a\in  X$ and $z\in  Z(X)\cap\frac ee \setminus\pi^{-1}[{\uparrow} F]$ such that $c=az^m$ for some $m>n$.
\end{claim}

\begin{proof} By our assumption, for every $n\in\IN$ there exists $\alpha\in X$ and $z\in  Z(X)\cap\frac ee \setminus\pi^{-1}[{\uparrow} F]$ such that $b=\alpha z^k$ for some $k>n$. Then
$$c=b^q\zeta =(\alpha z^k)^q\zeta=(\alpha^q\zeta)z^{kq}=az^m$$
where $a=\alpha^q\zeta$ and $m=kq\ge k>n$.
\end{proof}

\begin{claim}\label{cl:distinct-azn} Let  $a\in X$, $z\in  Z(X)\cap\frac ee \setminus\pi^{-1}[{\uparrow} F]$ and $n\in\IN$ be such that $c=az^n$. Then there exists $m\ge n$ such that $az^me\ne az^m$ and $az^{m+1}e=az^{m+1}$.
\end{claim}

\begin{proof}  Since $X$ is eventually Clifford, there exists a number $l>n$ such that $z^{l}\in H_{\pi(z)}$ and hence $z^l=z^l\pi(z)$. By Lemma~\ref{l:ZH}, $\pi(z)=\pi(z^l)\in EZ(X)$ and hence $\pi(z)u\in E(X)$. We claim that $az^le=az^l$. Indeed, assuming that $az^le\ne az^l$, we obtain that
\begin{multline*}
b^q\zeta z^{l-n}\pi(z)u=c z^{l-n}\pi(z)u=cuz^{l-n}\pi(z)=cz^{l-n}\pi(z)=\\
az^nz^{l-n}\pi(z)=az^l\pi(z)=az^l=az^nz^{l-n}=cz^{l-n}=b^q\zeta z^{l-n}=\\
cz^{l-n}=az^nz^{l-n}=az^{l}\ne az^le=az^nz^{l-n}e=cz^{l-n}e=b^q\zeta z^{l-n}e
\end{multline*}
and hence $\pi(z)u=u\pi(z)\in L$ and $\pi(z)u=u$ by the minimality of $u$. Then $z\in\pi^{-1}[{\uparrow} u]=\pi^{-1}[{\uparrow}F]$, which contradicts the choice of $z$. This contradiction shows that $az^le=az^l$.

On the other hand $az^ne=ce\ne c=az^n$. Then there exists $m\in\IN$ such that $n\le m<l$ and $az^me\ne az^m$ but $az^{m+1}e=az^{m+1}$.
\end{proof}

Using Claims~\ref{cl:exist-azn} and \ref{cl:distinct-azn}, we shall inductively  construct sequences $\{a_k\}_{k\in\w}\subseteq X$, $\{z_k\}_{k\in\w}\subseteq Z(X)\cap\frac ee\setminus\pi^{-1}[{\uparrow} F]$ and $(n_k)_{k\in\w},(l_k)_{k\in\w},(m_k)_{k\in\w}\in\IN^\w$ such that for every $k\in\w$ the following conditions are satisfied:
\begin{itemize}
\item[(i)] $a_kz_k^{n_k}=c$;
\item[(ii)] $a_kz_k^{l_k}e\ne a_kz_k^{l_k}$ and $a_kz_k^{l_k+1}e=a_kz_k^{l_k+1}$;
\item[(iii)] $a_kz_k^{m_k}\notin\{a_iz_i^{m_i}:i<k\}$;
\item[(iv)]  $2k+1<n_k\le l_k$ and $l_k-k\le m_k\le l_k$;
\item[(v)] $a_iz_i^{m_i}a_kz_k^{m_k}=c^2$ for any $i\le k$.
\end{itemize}

To start the inductive construction, apply Claim~\ref{cl:exist-azn} and find $a_0\in X$, $z_0\in Z(X)\cap\frac{e}e\setminus\pi^{-1}[{\uparrow} F]$ and $n_0>1$ such that $c=a_0z_0^{n_0}$. By Claim~\ref{cl:distinct-azn}, there exists a number $l_0\ge n_0$ such that $a_0z_0^{l_0}e\ne a_0z_0^{l_0}$ and  $a_0z_0^{l_0+1}e=a_0z_o^{l_0+1}$.  Put $m_0=l_0$. It is clear that $a_0,z_0,n_0,l_0,m_0$ satisfy the inductive conditions (i)--(iv). Let us show that (v) is satisfied, too. If $m_0=n_0$, then
$$(a_0z_0^{m_0})^2=(a_0z_0^{n_0})^2=c^2.$$
 If $n_0<m_0$, then
$$(a_0z_0^{m_0})^2=(a_0z_0^{n_0})^2z_0^{2(m_0-n_0)}=c^2z_0^{2(m_0-n_0)}=c^2e z_0^{2(m_0-n_0)}=c^2e=c^2,$$
because $z_0\in\frac ee$.
\smallskip

Assume that for some $k\in\w$ sequences $(a_i)_{i<k}$, $(z_i)_{i<k}$, $(n_i)_{i<k}$, $(l_i)_{i<k}$ and $(m_i)_{i<k}$ satisfying the inductive conditions (i)--(v) have been constructed.

\begin{claim} For every $i<k$ the set $A_i=\{a_iz_i^{m_i}z:z\in Z(X)\cap\frac ee\}$ is finite.
\end{claim}

\begin{proof} The inductive condition (iv) implies that $m_i\ge l_i-i>i+1$ and hence $l_i\ge 2i+2$ and $2m_i\ge (2l_i-2i)\ge l_i+2$.
The inductive conditions (ii), (iv) and $z_i\in Z(X)\cap\frac ee$ ensure that
\begin{multline*}
(a_iz_i^{m_i})^2=(a_iz_i^{l_i+1})a_iz_i^{2m_i-l_i-1}=(a_iz_i^{l_i+1}e)a_iz_i^{2m_i-l_i-1}=(a_iz_i^{m_i})^2e=
(a_iz_i^{m_i})^2z_i^{2l_i}e=\\
(a_iz_i^{n_i})^2z_i^{2l_i+2m_i-2n_i}e=c^2z_i^{2l_i+2m_i-2n_i}e=c^2e.
\end{multline*}
Then for every $z,z'\in Z(X)\cap\frac ee$ we have
 $$(a_iz_i^{m_i}z)(a_iz_i^{m_i}z')=(a_iz_i^{m_i})^2zz'=c^2ezz'=c^2zz'e=c^2e=c^2,$$
 which means that $A_iA_i\subseteq \{c^2\}$ and hence $A_i$ is finite by the nonsingularity of the semigroup $X$.
\end{proof}

By Claim~\ref{cl:exist-azn}, there exist $a_k\in X$, $z_k\in Z(X)\cap\frac ee\setminus\pi^{-1}[{\uparrow}F]$ and $n_k>2k+1+\max_{i<k}|A_i|$ such that $a_kz_k^{n_k}=c$.
By Claim~\ref{cl:distinct-azn}, there exists a number $l_k\ge n_k$ such that $a_kz_k^{l_k}e\ne a_kz_k^{l_k}$ but $a_kz_k^{l_k+1}e=a_kz_k^{l_k+1}$.

\begin{claim}\label{cl:rizni} For every positive numbers $i<j\le l_k$ we have $a_kz_k^i\ne a_kz_k^j$.
\end{claim}

\begin{proof} To derive a contradiction, assume that $a_kz_k^i=a_kz_k^j$ and hence
$$a_kz_k^{i+2(j-i)}=a_kz_k^jz_k^{j-i}=a_kz_k^iz_k^{j-i}=a_kz_k^j=z_kz_k^i.$$
Proceeding by induction, we obtain that $a_kz_k^{i+s(j-i)}=a_kz_k^i$ for every $s\in\IN$.
Find $s\in\IN$ such that $i+s(j-i)>l_k+1$. Then
\begin{multline*}
a_kz_k^{l_k}=a_kz_k^{i}z_k^{l_k-i}=a_kz_k^{i+s(j-i)}z^{l_k-i}=a_kz_k^{l_k+1}z_k^{i+s(j-i)-l_k-1+l_k-i}=\\
a_kz_k^{l_k+1}ez_k^{s(j-i)-1}=a_kz_k^{l_k}ez_k^{s(j-i)}=a_kz_k^{l_k}e,
\end{multline*}
which contradicts the choice of $l_k$.
\end{proof}

For every $i<k$, let
$$\lambda_i=\max\big(\{0\}\cup\{\lambda\in\{1,\dots,l_k\}:a_iz_i^{m_i}a_kz_k^{\lambda}e\ne a_iz_i^{m_i}a_kz_k^{\lambda}\}\big).$$

\begin{claim}\label{cl:lambda} $a_iz_i^{m_i}a_kz_k^{\lambda_i+1}e=a_iz_i^{m_i}a_kz_k^{\lambda_i+1}$.
\end{claim}

\begin{proof} If $a_iz_i^{m_i}a_kz_k^{\lambda_i+1}e\ne a_iz_i^{m_i}a_kz_k^{\lambda_i+1}$, then the definition of the number $\lambda_i$ ensures that $\lambda_i=l_k$. Then
$$a_iz_i^{m_i}a_kz_k^{\lambda_i+1}=a_iz_i^{m_i}a_kz_k^{l_k+1}=a_iz_i^{m_i}a_kz_k^{l_k+1}e=a_iz_i^{m_i}a_kz_k^{\lambda_i+1}e,$$
which contradicts our assumption.
\end{proof}

\begin{claim} $\lambda_i\le |A_i|$.
\end{claim}

\begin{proof} Assuming that $\lambda_i>|A_i|$, we conclude that $\lambda_i>0$ and hence $a_iz_i^{m_i}a_kz_k^{\lambda_i}e\ne a_iz_i^{m_i}a_kz_k^{\lambda_i}$. Then  also  $a_iz_i^{m_i}z_k^{\lambda_i}e\ne a_iz_i^{m_i}z_k^{\lambda_i}$ and $a_iz_i^{m_i}z_k^{j}e\ne a_iz_i^{m_i}z_k^{j}$ for every $j\le \lambda_i$. The definition of the set $A_i$ ensures that $\{a_iz_i^{m_i}z_k^j:j\in\IN\}\subseteq A_i$.
Since $\lambda_i>|A_i|$, there exist positive numbers $j<j'\le\lambda_i$ such that $a_iz_i^{m_i}z_k^j=a_iz_i^{m_i}z_k^{j'}=a_iz_i^{m_i}z_k^{j+(j'-j)}$.
Then $$a_iz_i^{m_i}z_k^{j+2(j'-j)}=a_iz_i^{m_i}z_k^{j'}z_k^{j'-j}=a_iz_i^{m_i}z_k^jz_k^{j'-j}=a_iz_i^{m_i}z_k^{j'}=a_iz_i^{m_i}z_k^j.$$Proceeding by induction, we can prove that
$$a_iz_i^{m_i}z_k^{j+s(j'-j)}=a_iz_i^{m_i}z_k^j$$for every $s\in\IN$. Choose $s\ge 2$ such that $j+s(j'-j)>\lambda_i+1$.
By Claim~\ref{cl:lambda},
\begin{multline*}
a_iz_i^{m_i}a_kz_k^{\lambda_i}=a_iz_i^{m_i}a_kz_k^jz_k^{\lambda_i-j}=a_iz_i^{m_i}a_kz_k^{j+s(j'-j)}z_k^{\lambda_i-j}=a_iz_i^{m_i}a_kz_k^{\lambda_i+1}z_k^{s(j'-j)-1}=\\
a_iz_i^{m_i}a_kz_k^{\lambda_i+1}ez_k^{s(j'-j)-1}=a_iz_i^{m_i}a_kz_k^{j+s(j'-j)}z_k^{\lambda_i-j}e=
a_iz_i^{m_i}a_kz_k^{j}z_k^{\lambda_i-j}e=a_iz_i^{m_i}a_kz_k^{\lambda_i}e,
\end{multline*}
which contradicts the definition of $\lambda_i$.
\end{proof}

Using Claim~\ref{cl:rizni}, choose a number $m_k$ such that $a_kz_k^{m_k}\notin\{a_iz_i^{m_i}:i<k\}$ and $l_k-k\le m_k\le l_k$. Observe that for every $i<k$ we have  $m_k\ge l_k-k\ge n_k-k>|A_i|\ge\lambda_i$ and hence $a_iz_i^{m_i}a_kz_k^{m_k}e=a_iz_i^{m_i}a_kz_k^{m_k}$, see Claim~\ref{cl:lambda}. Then
\begin{multline*}
a_iz_i^{m_i}a_kz_k^{m_k}=a_iz_i^{m_i}a_kz_k^{m_k}e=a_iz_i^{m_i}ea_kz_k^{m_k}e=a_iz_i^{m_i}z_i^{l_i}ea_kz_k^{m_k}z_k^{l_k}e=\\
a_iz_i^{n_i}z_i^{l_i+m_i-n_i}ea_kz_k^{n_k}z_k^{l_k+m_k-n_k}e=cece=c^2e=c^2.
\end{multline*}
Also $2m_k\ge 2l_k-2k\ge l_k+n_k-2k\ge l_k+2$ implies that
\begin{multline*}
(a_kz_k^{m_k})^2=a_k^2z^{2m_k}=a_k^2z_k^{l_k+1}z_k^{2m_k-l_k-1}=a_k^2z_k^{l_k+1}ez_k^{2m_k-l_k-1}=(a_kz_k^{m_k})^2e=(a_kz_k^{m_k})^2z^{2l_k}e=\\
(a_kz_k^{n_k})^2z_k^{2(l_k+m_k-n_k)}e=c^2e=c^2.
\end{multline*}
Therefore,  $a_k,z_k,n_k,l_k,m_k$ satisfy the inductive conditions (i)--(v).
\smallskip

After completing the inductive construction, consider the set $A=\{a_kz_k^{m_k}\}_{k\in\w}$. The inductive conditions  (iii) and (v) ensure that the set $A$ is infinite and $AA=\{c^2\}$ is a singleton. But this contradicts the nonsingularity of $X$.
\end{proof}

\section{Semigroup topologies generated by $e$-bases on semigroups}\label{s:topology}

In this section we introduce the notion of an $e$-base, which is a suitable simplification of the (more general) notion of an $e$-remote base.

\begin{definition}\label{d:remote}  Let $X$ be a semigroup and $e$ be a central idempotent in $X$. An {\em $e$-base} on $X$ is a nonempty family $\Phi$ of subsemigroups of $X$ such that for every $U,W\in\Phi$ there exists $V\in \Phi$ such that $V\subseteq U\cap W\subseteq Z(X)\cap \frac ee$.
\smallskip

Given an $e$-base $\Phi$ on $X$, let  $\Tau_\Phi$ be the topology on $X$, consisting of all sets $W\subseteq X$ such that for every $x\in W$ there exist a set $U\in\Phi$ such that $x\in\Lambdae(x;U)\subseteq W$. The topology $\Tau_\Phi$ will be referred to as {\em the topology generated by the $e$-base $\Phi$}.

The $e$-base $\Phi$ is {\em regular} if for every $b\in X$ with $b\ne be$ there exists a set $V\in\Phi$ such that $b\notin\frac{be}e\cdot V$.
\end{definition}

Observe that for an $e$-base $\Phi$ on a semigroup $X$, the constant function $\Phi_*$ assigning to each $x\in X$ the family $\Phi$ is an $e$-remote base on $X$ and the topology $\Tau_\Phi$ is equal to the topology $\Tau_{\Phi_*}$. The $e$-base $\Phi$ is regular if and only if the $e$-remote base $\Phi_*$ is regular.

The following theorem can be easily derived from Lemma~\ref{l:base-e}, Theorem~\ref{t:TS}, and Propositions~\ref{p:regular}, \ref{p:regular-powers}.

\begin{theorem}\label{t:TS-b} Let $X$ be a  semigroup, $e$ be a central idempotent in $X$ and $\Phi$ be an $e$-base on $X$. Then
\begin{enumerate}
\item $(X,\Tau_{\Phi})$ is a $T_0$ topological semigroup;
\item for every $x\in X$ the family $\{\Lambdae(x;U):U\in\Phi\}\subseteq\Tau_{\Phi}$ is a neighborhood base at $x$.
\item If the $e$-base $\Phi$ is regular, then the topological semigroup $(X,\Tau_{\Phi})$ is  zero-dimensional and for every $x\in X$ and $U\in\Phi$ the set $\Lambdae(x;U)$ is clopen in $(X,\Tau_\Phi)$.
\item The $e$-base $\Phi$ is regular if one of the following conditions is satisfied:
\begin{enumerate}
\item the poset $EZ(X)$ is well-founded and for every finite set $F\subseteq EZ(X)\setminus{\downarrow}e$, there exists a set $V\in\Phi$ such that $V\subseteq H(X)\cap Z(X)\setminus\pi^{-1}[{\uparrow}F]$;
\item the semigroup $X$ is nonsingular and eventually Clifford, the poset $E(X)$ is well-founded, and for every $n\in\IN$ and $F\in [E(X)\setminus{\downarrow}e]^{<\w}$, there exists a set $V\in\Phi$ such that $V\subseteq \{z^m:z\in  Z(X)\cap\frac ee \setminus\pi^{-1}[{\uparrow} F],\;m\ge n\}$.
\end{enumerate}
\end{enumerate}
\end{theorem}

Given a central idempotent $e$ in a semigroup $X$, consider the following families:
$$
\begin{aligned}
\mathcal E[e)&\defeq\big\{\tfrac ee\cap E(X)\cap Z(X)\setminus {\uparrow}F:F\in[EZ(X)\setminus {\downarrow}e]^{<\w}\big\};\\
\mathcal H[e)&\defeq\big\{\tfrac ee\cap H(X)\cap Z(X)\setminus \pi^{-1}[{\uparrow}F]:F\in[EZ(X)\setminus {\downarrow}e]^{<\w}\big\};\\
\mathcal Z[e)&\defeq\big\{\{z^{n}:z\in \tfrac ee\cap Z(X)\cap\pi^{-1}[EZ(X)\setminus {\uparrow}F]\}:n\in\IN,\;F\in [E(X)\setminus{\downarrow}e]^{<\w}\big\}.
\end{aligned}
$$

\begin{theorem}\label{t:TS-main} Let $X$ be a semigroup and $e$ be a central idempotent in $X$.
\begin{enumerate}
\item The families $\mathcal E[e)$, $\mathcal H[e)$ and $\mathcal Z[e)$ are $e$-bases for $X$ and hence $(X,\Tau_{\mathcal E[e)})$, $(X,\Tau_{\mathcal H[e)})$, and  $(X,\Tau_{\mathcal Z[e)})$ are $T_0$ topological semigroups.
\item If  the semilattice $EZ(X)$ is well-founded, then the families $\mathcal E[e)$ and $\mathcal H[e)$ are regular $e$-bases for $X$ and hence $(X,\Tau_{\mathcal E[e)})$ and $(X,\Tau_{\mathcal H[e)})$ are zero-dimensional  topological semigroups.
\item If the poset $E(X)$ is well-founded and the semigroup $X$ is nonsingular and eventually Clifford, then the family $\mathcal Z[e)$ is a regular $e$-base for $X$ and $(X,\Tau_{\mathcal Z[e)})$ is a zero-dimensional topological semigroup.
\end{enumerate}
\end{theorem}

\begin{proof} 1E. To see that the family $\mathcal E[e)$ is an $e$-base, take any sets $A_1,A_2\in\mathcal E[e)$ and find finite sets $F_1,F_2\subseteq EZ(X)\setminus{\downarrow}e$ such that $A_i=\frac ee\cap EZ(X)\setminus {\uparrow}F_i$ for every $i\in\{1,2\}$. Observe that for the finite set $F=F_1\cup F_2\subseteq EZ(X)\setminus{\downarrow}e$, the set $A\defeq \frac ee\cap EZ(X)\setminus {\uparrow}F$ belongs to $\mathcal E[e)$ and is a subset of $A_1\cap A_2$. Since $EZ(X)$ is a semilattice, the set $A$ is a subsemigroup of $EZ(X)$. Then $AA\subseteq A\subseteq A_1\cap A_2\subseteq Z(X)\cap\frac ee$, witnessing that the family $\mathcal E[e)$ is an $e$-base.
\smallskip

1H.  To see that the family $\mathcal H[e)$ is an $e$-base, take any sets $A_1,A_2\in\mathcal H[e)$ and find finite sets $F_1,F_2\subseteq EZ(X)\setminus{\downarrow}e$ such that $A_i=\frac ee\cap H(X)\cap Z(X)\setminus \pi^{-1}[{\uparrow}F_i]$ for every $i\in\{1,2\}$. Observe that for the finite set $F=F_1\cup F_2\subseteq EZ(X)\setminus{\downarrow}e$, the set $A\defeq \frac ee\cap H(X)\cap Z(X)\setminus \pi^{-1}[{\uparrow}F]$ belongs to $\mathcal H[e)$ and is a subset of $A_1\cap A_2$. We claim that $A$ is a subsemigroup of $X$. By Lemma~\ref{l:ZH}, $H(X)\cap Z(X)=H(Z(X))$, by Lemma~\ref{l:inverse}, $H(Z(X))$ is a subsemigroup of $Z(X)$, and by Lemma~\ref{l:pi-homo}, the restriction $\pi{\restriction}_{H(Z(X))}:H(Z(X))\to EZ(X)$ is a homomorphism. Since the set $EZ(X)\setminus{\uparrow}F$ is a subsemigroup of $EZ(X)$ and $\pi{\restriction}_{H(Z(X))}$ is a homomorphism, the set $H(Z(X))\cap\pi^{-1}[{\uparrow}F]$ is a subsemigroup of $H(Z(X))$. It is easy to see that $\frac ee$ is a subsemigroup of $X$. Then $A=\frac ee\cap H(Z(X))\setminus\pi^{-1}[{\uparrow}F]$ is a subsemigroup of $H(Z(X))\subseteq Z(X)$ and hence  $AA\subseteq A\subseteq A_1\cap A_2\subseteq Z(X)\cap\frac ee$, witnessing that the family $\mathcal H[e)$ is an $e$-base.
\smallskip

1Z.   To see that the family $\mathcal Z[e)$ is an $e$-base, take any sets $A_0,A_1\in\mathcal Z[e)$ and find numbers $n_0,n_1\in\IN$ and finite sets $F_0,F_1\subseteq E(X)\setminus{\downarrow}e$ such that $A_i=\{z^{n_i}:z\in  Z(X)\cap\frac ee \cap\pi^{-1}[EZ(X)\setminus {\uparrow}F_i]\}$ for every $i\in\{0,1\}$. Observe that for the number $n=n_0n_1$ and the finite set $F=F_0\cup F_1\subseteq E(X)\setminus{\downarrow}e$, the set $A\defeq \{z^n:z\in  Z(X)\cap\frac ee \cap \pi^{-1}[EZ(X)\setminus {\uparrow}F]\}$ belongs to $\mathcal Z[e)$. We claim that $A\subseteq A_0\cap A_1$. Take any $a\in A$ and find $z\in Z(X)\cap\frac ee \cap\pi^{-1}[EZ(X)\setminus{\uparrow}F]$ such that $a=z^n$. It follows from $z\in \pi^{-1}[EZ(X)\setminus{\uparrow}F]\subseteq\korin{\infty}{H(X)}$ that for every $i\in \{0,1\}$ we have $z^{n_i}\in\korin{\infty}{H(X)}$ and $\pi(z^{n_i})=\pi(z)\in EZ(X)\setminus{\uparrow}F$. Since the set $ Z(X)\cap\frac ee $ is a subsemigroup of $X$, we obtain $z^{n_i}\in  Z(X)\cap\frac ee \cap\pi^{-1}[EZ(X)\setminus{\uparrow}F]\subseteq Z(X)\cap\frac ee \cap\pi^{-1}[EZ(X)\setminus{\uparrow}F_i]$. Taking into account that  $a=z^n=z^{n_0n_1}=(z^{n_i})^{n_{1-i}}$, we conclude that $a\in A_{1-i}$ for every $i\in\{0,1\}$. 

It remains to show that $A$ is a subsemigroup of $Z(X)$.
Given any elements $a,b\in A$, find elements $x,y\in  Z(X)\cap\frac ee \cap\pi^{-1}[EZ(X)\setminus{\uparrow}F\}$ such that $a=x^n$ and $b=y^n$. Then $ab=x^ny^n=(xy)^n$ as $x,y\in Z(X)$. It remains to prove that $xy\in  Z(X)\cap\frac ee \cap \pi^{-1}[EZ(X)\setminus {\uparrow}F]$. It follows from $x,y\in  Z(X)\cap\frac ee $ that $xy\in Z(X)\cap\frac ee $. By Lemma~\ref{l:ZH}, we have $H(X)\cap Z(X)=H(Z(X))$ and hence $x,y\in Z(X)\cap \pi^{-1}[EZ(X)\setminus {\downarrow}F]\subseteq Z(X)\cap \korin{\infty}{H(X)}=Z(X)\cap\korin{\infty}{H(Z(X))}$. By Lemma~\ref{l:pi-homo}, the set $Z(X)\cap \korin{\infty}{H(Z(X))}$ is a subsemigroup of $Z(X)$ and the restriction $\pi{\restriction}_{\korin{\infty}{H(Z(X))}}$ is a homomorphism. Then $xy\in Z(X)\cap \korin{\infty}{H(Z(X))}$ and $\pi(xy)=\pi(x)\pi(y)\in EZ(X)$. Assuming that $\pi(xy)\in{\uparrow}F$, we can find an idempotent $f\in F$ such that $\pi(xy)f=f\pi(xy)=f$ and then
$$f=\pi(xy)f=\pi(x)\pi(y)f=\pi(x)\pi(x)\pi(y)f=\pi(x)f=f\pi(x)$$as $\pi(x)\in EZ(X)$. Then $\pi(x)\in EZ(X)\cap {\uparrow}f$, which contradicts the choice of $x\in \pi^{-1}[EZ(X)\setminus{\uparrow}F]\subseteq \pi^{-1}[EZ(X)\setminus{\uparrow}f]$. This contradiction shows that $\pi(xy)\notin{\uparrow}F$ and hence $xy\in Z(X)\cap\frac ee \cap\pi^{-1}[EZ(X)\setminus{\uparrow}F]$ and $ab=(xy)^n\in A$.

Then $AA\subseteq A\subseteq A_1\cap A_2\subseteq Z(X)\cap\frac ee$, witnessing that $\mathcal Z[e)$ is an $e$-base.
\smallskip

By Theorem~\ref{t:TS-b}, $(X,\Tau_{\mathcal E[e)})$,  $(X,\Tau_{\mathcal H[e)})$ and  $(X,\Tau_{\mathcal Z[e)})$ are $T_0$ topological semigroups.
\smallskip

2. If the poset $EZ(X)$ is well-founded, then by Theorem~\ref{t:TS-b}(4a), the $e$-bases $\mathcal E[e)$ and $\mathcal H[e)$ are regular and by Theorem~\ref{t:TS-b}(3), the $T_0$ semigroup topologies  $\Tau_{\mathcal E[e)},\Tau_{\mathcal H[e)}$ are zero-dimensional.
\smallskip

3. If the poset $E(X)$ is well-founded and the semigroup $X$ is nonsingular and eventually Clifford, then by Theorem~\ref{t:TS-b}(4b), the $e$-base $\mathcal Z[e)$ is regular and by Theorem~\ref{t:TS-b}(3), the $T_0$ semigroup topology  $\Tau_{\mathcal Z[e)}$ is  zero-dimensional.
\end{proof}

Now we find a condition ensuring that the topologies $\Tau_{\mathcal E[e)}$, $\Tau_{\mathcal H[e)}$, $\Tau_{\mathcal Z[e)}$ are not discrete.

\begin{proposition}\label{p:nonisolated} If the central semilattice $EZ(X)$ of a semigroup $X$ is chain-finite and infinite, then some idempotent $e\in EZ(X)$ is a non-isolated point in the topologies $\Tau_{\mathcal E[e)}$, $\Tau_{\mathcal H[e)}$, and $\Tau_{\mathcal Z[e)}$.
\end{proposition}

\begin{proof} Assume that the semilattice $EZ(X)$ is chain-finite and infinite.   Observe that the set $$I\defeq\{e\in EZ(X):\mbox{$EZ(X)\cap {\uparrow}e$ is infinite}\}$$ contains the smallest element of the chain-finite semilattice $EZ(X)$ and hence $I$ is not empty. Since $EZ(X)$ is chain-finite, there exists an idempotent $e\in I$ such that $I\cap {\uparrow}e=\{e\}$. We claim that $e$ is non-isolated in the topologies $\Tau_{\mathcal E[e)}$, $\Tau_{\mathcal H[e)}$, and $\Tau_{\mathcal Z[e)}$. Indeed, given any neighborhood $O_e$ of $e$ in one of these topologies, we can find a finite set $F\subseteq E(X)\setminus{\downarrow}e$ such that $\frac ee\cap EZ(X)\setminus{\uparrow}F\subseteq O_e$. It remains to prove that the set $\frac ee\cap EZ(X)\setminus{\uparrow}F\subseteq O_e$ is infinite. First we show that for every $f\in F$, the set $S_f\defeq EZ(X)\cap {\uparrow}e\cap{\uparrow}f$ is finite. This is clear if $S_f$ is empty. So, assume that $S_f\ne\emptyset$. Being a subsemilattice of the chain-finite semilattice $EZ(X)$, the semilattice $S_f$ has the smallest element $s\in S_f\subseteq {\uparrow}e\cap{\uparrow}f$. Assuming that $s=e$, we conclude that $f\in{\downarrow}s={\downarrow}e$, which contradicts the choice of $f\in F\subseteq E(X)\setminus {\downarrow}e$. Therefore, $s\in{\uparrow}e\setminus\{e\}$. The maximality of $e$ ensures that the set $EZ(X)\cap {\uparrow}s\supseteq S_f$ is finite. Then the set $EZ(X)\cap\frac ee\setminus{\uparrow}F=EZ(X)\cap {\uparrow}e\setminus{\uparrow}F=EZ(X)\cap {\uparrow}e\setminus\bigcup_{f\in F}S_f$ is infinite and so is the set $O_e\supseteq EZ(X)\cap\tfrac ee\setminus {\uparrow}F$.
\end{proof}

\begin{theorem}\label{t:topologizable} A semigroup $X$ is $\mathsf{T_{\!z}S}$-topologizable if its central semilattice $EZ(X)$ is chain-finite and infinite.
\end{theorem}

\begin{proof} Assume that the semilattice $EZ(X)$ is chain-finite and infinite. Since the semilattice $EZ(X)$ is chain-finite, the poset $EZ(X)$ is well-founded. By Proposition~\ref{p:nonisolated}, some idempotent $e$ is non-isolated in the topology $\Tau_{\mathcal E[e)}$. By Theorem~\ref{t:TS-main}(2), $(X,\Tau_{\mathcal E[e)})$ is a Hausdorff zero-dimensional topological semigroup, witnessing that the semigroup $X$ is $\mathsf{T_{\!z}S}$-topologizable.
\end{proof}






\section{Proof of Theorem~\ref{t:main-iT1}}\label{s:main-iT1}

Given a semigroup $X$, we need to prove the implications $(1)\Ra(2)\Leftrightarrow(3)\Ra(4)\Ra(5)\Ra(6)\Leftrightarrow(7)$ of the conditions:
\begin{enumerate}
\item $X$ is commutative, bounded, nonsingular and Clifford-finite;
\item $X$ is injectively $\mathsf{T_{\!1}S}$-closed;
\item $X$ is $\mathsf{T_{\!1}S}$-closed and $\mathsf{T_{\!1}S}$-discrete;
\item $X$ is $\mathsf{T_{\!z}S}$-closed and $\mathsf{T_{\!z}S}$-discrete;
\item $X$ is $\mathsf{T_{\!z}S}$-closed and $Z(X)$ is Clifford-finite;
\item $Z(X)$ is bounded, nonsingular and Clifford-finite;
\item $Z(X)$ is injectively $\mathsf{T_{\!1}S}$-closed.
\end{enumerate}
\bigskip

$(1)\Ra(2)$: Assume that $X$ is commutative, bounded, nonsingular and Clifford-finite. To prove that $X$ is injectively $\mathsf{T_{\!1}S}$-closed, it suffices to show that $X$ is closed in any $T_1$ topological semigroup $Y$ that contains $X$ as a subsemigroup. Let $\pi:X\to E(X)$ be the map assigning to each $x\in X$ the unique idempotent $\pi(x)\in E(X)$ such that $x\in H_{\pi(x)}$. By Lemma~\ref{l:pi-homo}, $\pi$ is a homomorphism and hence for every $e\in E(X)$ the preimage $\pi^{-1}(e)$ is a unipotent semigroup. By Theorem~\ref{t:unipotent}, the bounded, nonsingular, group-finite commutative unipotent semigroup $\pi^{-1}(e)$ is closed in $Y$. Since the set $E(X)\subseteq H(X)$ is finite, the set $X=\bigcup_{e\in E(X)}\pi^{-1}(e)$ is closed in $Y$, being the union of finitely many closed sets.
\smallskip

The equivalence $(2)\Leftrightarrow (3)$ is proved in Proposition 3.2 of \cite{BB2}.
\smallskip

The implication $(3)\Ra(4)$ is trivial and follows immediately from the inclusion $\mathsf{T_{\!z}S}\subseteq\mathsf{T_{\!1}S}$.
\smallskip

$(4)\Ra(5)$: Assume that $X$ is $\mathsf{T_{\!z}S}$-closed and $\mathsf{T_{\!z}S}$-discrete. By Theorem~\ref{t:center}, the central semilattice $EZ(X)$ is chain-finite and by Theorem~\ref{t:topologizable}, $EZ(X)$ is finite. By Theorem~\ref{t:center}, the semigroup $Z(X)$ is group-finite. By Lemma~\ref{l:ZH}, for every idempotent $e\in EZ(X)$ the intersection $H_e\cap Z(X)$ is a subgroup of $Z(X)$. Since $Z(X)$ is group-finite, the subgroup $H_e\cap Z(X)$ is finite. Then $H(Z(X))=\bigcup_{e\in EZ(X)}(H_e\cap Z(X))$ is finite, which means that the semigroup $Z(X)$ is Clifford-finite.
\smallskip

$(5)\Ra(6)$: Assume that $X$ is $\mathsf{T_{\!z}S}$-closed and $Z(X)$ is Clifford-finite. By Theorem~\ref{t:center}, the semigroup $Z(X)$ is periodic and nonsingular.
Let $\pi_Z:Z(X)\to EZ(X)$ be the map assigning to each $x\in X$ the unique idempotent $\pi_Z(x)\in x^\IN$. By Lemma~\ref{l:pi-homo}, the map $\pi_Z$ is a homomorphism. Then for every $e\in EZ(X)$ the preimage $\pi_Z^{-1}(e)\subseteq Z(X)$ is a unipotent commutative semigroup. Since $Z(X)$ is periodic, nonsingular and Clifford-finite, for every $e\in EZ(X)$, the unipotent commutative semigroup $\pi_Z^{-1}(e)$ is chain-finite, periodic, nonsingular and group-finite. By Theorem~\ref{t:C-closed}, $\pi^{-1}(e)$ is $\mathsf{T_{\!1}S}$-closed and by Theorem~\ref{t:unipotent-C}, $\pi_Z^{-1}(e)$ is bounded.  Since $Z(X)$ is Clifford-finite, the semilattice $EZ(X)$ is finite and then the semigroup $Z(X)=\bigcup_{e\in EZ(X)}\pi_Z^{-1}(e)$ is bounded, being a finite union of bounded subsemigroups.
\smallskip

The equivalence $(6)\Leftrightarrow(7)$ follows from the implications $(1)\Ra(2)\Ra(6)$ and the obvious equality $Z(Z(X))=Z(X)$.

\section{Proof of Theorem~\ref{t:main-aT1}}\label{s:main-aT1}

Given a semigroup $X$, we need to prove the implications $(1)\Ra(2)\Leftrightarrow(3)\Ra(4)\Leftrightarrow(5)\Ra(6)\Ra(7)\Ra(8)\Leftrightarrow(9)$ of the conditions:
\begin{enumerate}
\item $X$ is finite;
\item $X$ is absolutely $\mathsf{T_{\!1}S}$-closed;
\item $X$ is projectively $\mathsf{T_{\!1}S}$-closed and projectively $\mathsf{T_{\!1}S}$-discrete;
\item $X$ is projectively $\mathsf{T_{\!1}S}$-closed and injectively $\mathsf{T_{\!1}S}$-closed;
\item $X$ is projectively $\mathsf{T_{\!1}S}$-closed and $\mathsf{T_{\!1}S}$-discrete;
\item $X$ is ideally $\mathsf{T_{\!z}S}$-closed and $\mathsf{T_{\!z}S}$-discrete;
\item $X$ is ideally $\mathsf{T_{\!z}S}$-closed and the semigroup $Z(X)$ is Clifford-finite;
\item $Z(X)$ is finite;
\item $Z(X)$ is absolutely $\mathsf{T_{\!1}S}$-closed.
\end{enumerate}
\bigskip

The implication $(1)\Ra(2)$ is trivial, the equivalence $(2)\Leftrightarrow(3)$ is proved in Proposition 3.3 of \cite{BB2}, the implication $(2)\Ra(4)$ is trivial, the equivalence $(4)\Leftrightarrow(5)$ follows from the equivalence $(2)\Leftrightarrow(3)$ in Theorem~\ref{t:main-iT1}; the implication $(5)\Ra(6)$ is trivial,  and $(6)\Ra(7)$ follows from the implication $(4)\Ra(5)$ of Theorem~\ref{t:main-iT1}.
\smallskip

$(7)\Ra(8)$: Assume that $X$ is ideally $\mathsf{T_{\!z}S}$-closed and $Z(X)$ is Clifford-finite. By Theorem~\ref{t:center}, the semigroup $Z(X)$ is periodic. Let $\pi_Z:Z(X)\to EZ(X)$ be the map assigning to each $x\in Z(X)$ the unique idempotent in the monogenic semigroup $x^\IN$. By Lemma~\ref{l:pi-homo}, $\pi_Z$ is a homomorphism and hence for every $e\in EZ(X)$ the preimege $\pi_Z^{-1}(e)=\korin{\infty}{e}$ is a unipotent subsemigroup of $Z(X)$. By Lemma~\ref{l:ZH}, the intersection $H_e\cap Z(X)$ is a subgroup of $Z(X)$. Since $Z(X)$ is Clifford-finite, the subgroup $H_e\cap Z(X)$ is finite. Since $X$ is ideally $\mathsf{T_{\!z}S}$-closed, we can apply Lemma~7.5 of \cite{BB} and conclude that the set $\pi_Z^{-1}(e)\setminus H_e$ is finite and so is the semigroup $\pi^{-1}(e)=(\pi_Z^{-1}(e)\setminus H_e)\cap (H_e\cap Z(X))$. Then the set $Z(X)=\bigcup_{e\in EZ(X)}\pi_Z^{-1}(e)$ is finite.
\smallskip

The equivalence $(8)\Leftrightarrow(9)$ follows from the implications $(1)\Ra(2)\Ra(8)$ and the trivial equality $Z(Z(X))=Z(X)$.

\end{document}